\documentclass[11pt]{article}
\usepackage{latexsym}
\usepackage{amsmath}
\usepackage{amssymb,amsfonts,amscd}
\usepackage{multirow}
\usepackage{mathabx} 
\usepackage{hyperref}
\usepackage{color}
\usepackage{subcaption}
\usepackage{cite}
\usepackage{algorithm}
\usepackage{algorithmic}
\usepackage{amssymb}

\usepackage{subcaption}
\usepackage{graphicx}

\newtheorem{theorem}{Theorem}[section]

\newtheorem{rem}[theorem]{Remark}

\newtheorem{example}[theorem]{Example}
\newtheorem{mydef}[theorem]{Definition}

\newcommand\qed{{\hspace*{\fill}$\Box$\vskip12pt plus 1pt}}
\newcommand\mymapsto{\mathrel{\ooalign{$\rightarrow$\cr%
  \kern-.15ex\raise.275ex\hbox{\scalebox{1}[0.522]{$\mid$}}\cr}}}

\newenvironment{proof}{{\noindent\bf Proof.\ }}{\qed}
\newenvironment{remark}{\begin{rem}\em}{\end{rem}}

\newcommand{\sG}{{\mathcal G}}

\newcommand{\sL}{{\mathcal L}}

\newcommand\bC{{\mathbb C}}
\newcommand\bN{{\mathbb N}}

\newcommand\bR{{\mathbb R}}
\newcommand\bZ{{\mathbb Z}}

\newcommand\cS{{\mathcal S}}

\newcommand\suchthat{~|~}

\DeclareMathOperator*{\Pow}{{\rm Pow}}
\DeclareMathOperator*{\OPow}{{\rm OPow}}
\DeclareMathOperator*{\IPow}{{\rm IPow}}

\begin{document}
\title{Real monodromy action}

\author{
Jonathan D. Hauenstein\thanks{Department of Applied and Computational Mathematics and Statistics,
University of Notre Dame, Notre Dame, IN 46556 (hauenstein@nd.edu, \url{www.nd.edu/\~jhauenst}).
This author was supported in part by
NSF grant CCF-1812746 and
ONR N00014-16-1-2722.
}
\and
Margaret H. Regan\thanks{Department of Applied and Computational Mathematics and Statistics,
University of Notre Dame (mregan9@nd.edu,
\url{www.nd.edu/\~mregan9}).
This author was supported in part by
Schmitt Leadership Fellowship in Science and Engineering and NSF grant CCF-1812746.}
}

\date{\today}

\maketitle

\begin{abstract}
\noindent
The monodromy group
is an invariant for parameterized systems
of polynomial equations that encodes
structure of the solutions over the parameter space.
Since the structure of real solutions over
real parameter spaces are of interest
in many applications, real monodromy action
is investigated here.
A naive extension of monodromy action
from the complex numbers to the real numbers
is shown to be very restrictive.
Therefore, we define
a real monodromy structure which need not be a group
but contains tiered characteristics about the real
solutions.  This real monodromy structure is applied to an example in kinematics which summarizes
all the ways performing loops parameterized by
leg lengths can cause a mechanism to
change poses.


\end{abstract}


\section{Introduction}\label{Sec:Intro}

For a polynomial system defined over
a complex parameter space, the
monodromy group encodes
permutations of the solutions
over loops in the parameter space
and can be viewed as a geometric counterpart
to Galois groups \cite{Harris,Hermite}
utilized in number theory and arithmetic geometry.
Monodromy groups are used in algebraic geometry
to view structure of the solutions such as
symmetry, restrictions on the number of
real solutions, and decomposition of
varieties into irreducible components.
The complex numbers bestow many properties
on the monodromy group such as it is
base point independent and does
not change when restricting to a general
curve section of the parameter space \cite{Zariski}.
These simplify the
computation of the monodromy group \cite{MonodromyAlg}
summarized in \S~\ref{Sec:CMG}.

Since real solutions over real points in a parameter
space are typically of most interest in many
applications, we aim to understand the behavior
of the real solutions over real loops in the
parameter space.  In kinematics, this is related to
nonsingular assembly mode change for parallel manipulators \cite{Gosselin,Hayes,Zein,Innocenti,Macho,Bonev,Husty}
which is important in calibration due to
the possible change of pose at the ``home'' position.
To illustrate, consider the 3RPR mechanism shown in Figure~\ref{Fig:mechanism} which consists of
three prismatic legs with revolute joints
that are anchored on one side and
attached to a moving triangular platform on the other.
Thus, one would need to identify all the ways that real motion of the mechanism can lead to different poses
in the ``home'' position defined by a
fixed set of leg~lengths.
In this context,
the main motivation is to develop
a mathematical description of all
possible nonsingular assembly mode changes
described by real monodromy action.

\begin{figure}[h!]
	\centering
	\includegraphics[scale = 0.5]{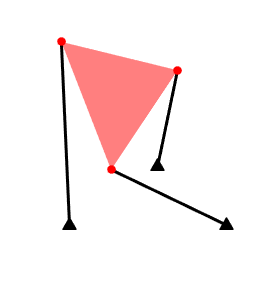}
	\caption{An example of a 3RPR mechanism.}\label{Fig:mechanism}
\end{figure}

A naive approach is to utilize a real monodromy
group defined similarly as the (complex) monodromy group. This definition leads to heavy restrictions on the construction of real loops and can often cause no pertinent information to be gained in the computations as discussed in~\S~\ref{Sec:RMG_def1}.
We instead propose a {\em real monodromy structure}
to obtain piecewise information about the permutations of the real solutions outlined in \S~\ref{Sec:RMG_def2}.  In particular, this
real monodromy structure contains all information
regarding nonsingular assembly mode changes.

The remainder of the paper is as follows. Section~\ref{Sec:CMG} describes some background
on the (complex) monodromy group.
The real monodromy group and structure are described in Section~\ref{Sec:RMG}
along with illustrative examples.
Section~\ref{Sec:3RPR} describes computing
the real monodromy structure of a 3RPR
mechanism where 2 of the legs can change
length.  Finally, the paper concludes in Section~\ref{Sec:Conclusion}.

\section{Monodromy group}\label{Sec:CMG}

Let $F(x;p)$ be a polynomial system with variables
$x\in\bC^N$ and parameters $p\in\bC^P$.
Assume that $F(x;p^*)=0$ has $D\in\bN$
isolated nonsingular solutions in $\bC^N$
for generic $p^*\in\bC^P$.  That is,
there is a Zariski open dense subset $U\subset\bC^P$
such that $F(x;p^*) = 0$ has $D$ nonsingular
isolated solutions in~$\bC^N$ for every $p^*\in U$.
Fix a point $b\in U$ and let $x^{(1)},\dots,x^{(D)}\in\bC^N$ be the $D$ nonsingular isolated solutions
to $F(x;b) = 0$.
If $\cS_D$ denotes the symmetric
group on $D$ elements,
then each loop $\gamma\subset U$
starting and ending at $b$ generates a permutation
$\sigma_\gamma \in \cS_D$
where $\sigma_\gamma(i) = j$
provided that the solution path of $F(x;p) = 0$
over $\gamma$ starting at $x^{(i)}$
ends at $x^{(j)}$.
The corresponding {\em monodromy group} is
simply the collection of all such permutations, namely
\begin{equation}\label{eq:ComplexMonodromyGroup}
\{\sigma_\gamma\in \cS_D\suchthat\gamma\subset U\hbox{~is loop starting and ending at~}b\}.
\end{equation}
The group structure arises naturally from concatenation
of loops.

The monodromy group is independent of the choice
$b\in U$ and the monodromy group of $F(x;p)$
for $p\in\bC^P$ is equal to the monodromy
group of $G(x;t) = F(x;\ell(t))$
where $\ell:\bC\rightarrow\bC^P$ is a general
affine linear function \cite{Zariski}.
In the one parameter case,
the monodromy group
is generated by the permutations arising
from the finitely many loops
that generate the fundamental group of
the intersection of $U$ to the line parameterized
by $\ell(t)$.  This is described
in detail with numerical algebraic geometric
computations in \cite{MonodromyAlg}
and illustrated in the following example.

\begin{example}\label{ex:Notes1.2}
Consider the parameterized
polynomial system
$$F(x;p) = \left[\begin{array}{c} x_1^2-x_2^2-p_1 \\ 2x_1x_2 - p_2 \end{array}\right] = 0$$
and $U = \{p\in\bC^2\suchthat p_1^2+p_2^2\neq 0\}$.
Thus, for every $p^*\in U$, $F(x;p^*) = 0$
has $D = 4$ nonsingular isolated solutions in $\bC^2$.
We take $b = (1,0)\in U$ with corresponding solutions
$$
x^{(1)} = (1,0),\hspace{5mm} x^{(2)} = (-1,0),\hspace{5mm} x^{(3)} = (0,\sqrt{-1}),\hspace{5mm} x^{(4)} = (0,-\sqrt{-1}).
$$

\begin{figure}[!b]
\begin{center}
\begin{tabular}{cc}
	\includegraphics[scale = 0.5]{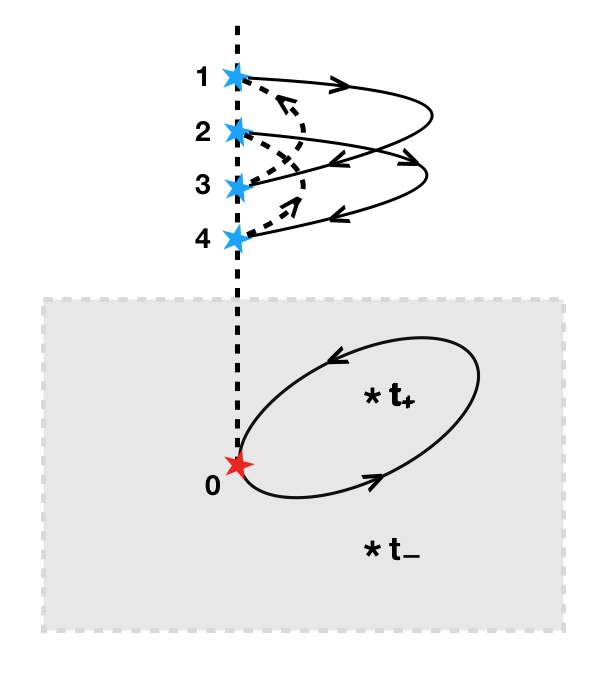} &
    \includegraphics[scale = 0.5]{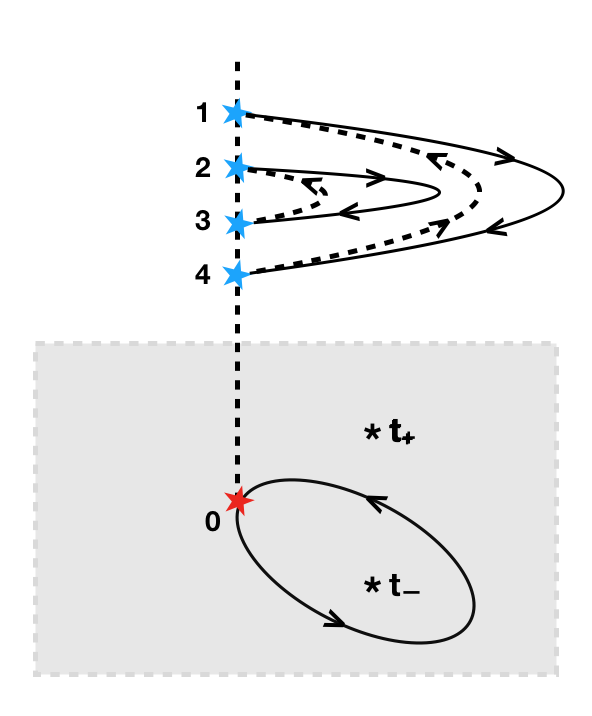} \\
    (a) & (b)
    \end{tabular}
    \end{center}
\caption{Illustration of (a) loop $\gamma_+$
and permutation $\sigma_{\gamma_+}$
and (b) loop $\gamma_-$
and permutation~$\sigma_{\gamma_-}$}
    \label{Fig:IllustrativeMonodromy}
\end{figure}

Consider restricting the parameter space to the line $\sL\subset\bC^2$ parameterized by $\ell(t) = (1-t,2t)$
so that \mbox{$\ell(0) = b$}.
In particular, $U\cap\sL\cong \bC\setminus\{t_+,t_-\}$
where $t_{\pm} = \frac{1\pm 2\sqrt{-1}}{5}$.
Let $\gamma_\pm\subset\bC$
be a simple loop starting and ending at~$0$ which encircles $t_\pm$ but not $t_\mp$, respectively.
The corresponding permutations, written in cycle notation, are
$$\sigma_{\gamma_+} = (1\hspace{1.5mm}3) (2\hspace{1.5mm}4)
\hbox{~~~~~~and~~~~~~}
\sigma_{\gamma_-} = (1\hspace{1.5mm}4) (2\hspace{1.5mm}3)
$$
which are illustratively shown in Fig.~\ref{Fig:IllustrativeMonodromy}.
Therefore, the monodromy group
is generated by the two permutations $\sigma_{\pm}$
which is the Klein group on four elements
$K_4 = \bZ_2\times\bZ_2\subset\cS_4$, namely
\begin{equation}\label{eq:K4}
K_4 = \{(1),\hspace{2mm} (1\hspace{1.5mm}2)(3\hspace{1.5mm}4),\hspace{2mm} (1\hspace{1.5mm}3)(2\hspace{1.5mm}4),\hspace{2mm} (1\hspace{1.5mm}4)(2\hspace{1.5mm}3)\}.
\end{equation}
The reason for the defectivity of the monodromy
group can be observed, for example, from the relation
between $x_1$, $p_1$, and $p_2$, namely
$$4x_1^4 - 4x_1^2p_1 - p_2^2 = 0$$
showing that the solutions arise in two pairs since
$$x_1 = \pm\sqrt{\frac{p_1\pm \sqrt{p_1^2+p_2^2}}{2}}.$$
\end{example}

We note that if the monodromy group is the full symmetric group,
computing corresponding permutations for a
few random loops is generally sufficient
to generate the full symmetric group~\cite{Leykin}.
Random monodormy loops can also be used effectively for performing numerical irreducible decomposition \cite{BHSW:BertiniBook,Sommese}.

\section{Monodromy over the real numbers}\label{Sec:RMG}

Since many applications, particularly ones arising in science and engineering, are only interested in the real solutions,
we consider monodromy action of real solutions
over a real parameter space for a real polynomial system
$F(x;p)$ where $x\in\bR^N$ and $p\in\bR^P$.

\subsection{Real monodromy group}\label{Sec:RMG_def1}

In the complex case, as summarized in \S~\ref{Sec:CMG},
loops arise on the nonempty Zariski open dense subset
$U\subset\bC^P$ consisting of parameter values
where the number of nonsingular isolated complex solutions is constant.
One can naively take a similar approach over the real numbers
as follows.  Fix a base point $b\in U\cap \bR^P$
and let $R$ be the number of nonsingular isolated
real solutions to $F(x;b) = 0$.  Hence,
there is a connected open subset $U_R\subset\bR^P$
containing $b$ such that the number of nonsingular
isolated real solutions of $F(x;p^*) = 0$ is equal
to $R$ for all $p^*\in U_R$.
As in complex case, loops in~$U_R$ yield
permutations in $\cS_R$.
Following \eqref{eq:ComplexMonodromyGroup},
the collection of all such permutations
\begin{equation}\label{eq:RealMonodromyGroup}
\{\sigma_\gamma\in \cS_R\suchthat\gamma\subset U_R\hbox{~is loop starting and ending at~}b\}
\end{equation}
is the {\em real monodromy group} of $F(x;p)$ with base
point $b$.  This group is naturally independent of the choice
of base point $b$ inside of $U_R$.

\begin{example}\label{ex:Notes2.1/2.2}
Reconsider $F(x;p)$ from Ex.~\ref{ex:Notes1.2}
with $b = (1,0)$.  There are $R = 2$ real solutions
to $F(x;b) = 0$, namely $x^{(1)} = (1,0)$ and $x^{(2)} = (-1,0)$.
Since, for all $p^*\in U_R = \bR^2\setminus\{(0,0)\}$,
$F(x;p^*) = 0$ has $2$ real solutions, the corresponding
real monodromy group
is generated by the loop $\gamma:[0,1]\rightarrow U_R$ where $\gamma(t) = (\cos(2\pi t),\sin(2\pi t))$
with $\gamma(0) = \gamma(1) = b$.
This loops yields the permutation $\sigma_\gamma = (1\hspace{1.5mm}2)$ so that the corresponding real monodromy group
is $\cS_2 = \{(1), (1\hspace{1.5mm}2)\}$.
\end{example}

In the complex case, there is a unique monodromy group
based on the nonempty Zariski open dense subset $U\subset\bC^P$.
In the real case, there can be multiple real monodromy groups.
The number of real solutions $R$ on the open set $U_R$
is called a {\em typical} number of
nonsingular real solutions of $F(x;p) = 0$.
Thus, one has a real monodromy group of $F(x;p)$
for each typical number of nonsingular
real solutions $R$
and each connected open subset $U_R$ in the set of all parameters
$p^*\in\bR^P$ for which $F(x;p^*) = 0$ has $R$
nonsingular isolated real solutions.

\begin{example}\label{ex:Univariate}
The polynomial $F(x;p) = x^2 + 1 - p^2$ has two typical
number of real solutions, namely $0$ when $-1 < p < 1$ and
$2$ when either $p < -1$ or $p > 1$.
The real monodromy group for $R = 0$ and $U_R = (-1,1) \subset\bR$
is empty since there are no real solutions.
The real monodromy group for $R = 2$ for both $U_R = (-\infty,-1) \subset\bR$ and $U_R = (1,\infty) \subset\bR$
is $\{(1)\}\subset\cS_2$.
\end{example}

The previous illustrative examples suggest the following.

\begin{theorem}\label{Thm:OnePar_RMG}
Suppose that $F(x;p)$ is a real polynomial system
with $x\in\bR^N$ and $p\in\bR^P$.  Let $b\in\bR^P$
such that $F(x;b) = 0$ has $R>0$ nonsingular isolated real
solutions and $U_R\subset\bR^P$ which is
a connected open set containing $b$
such that the number of nonsingular
isolated real solutions of $F(x;p^*) = 0$ is equal
to $R$ for all $p^*\in U_R$.
If the fundamental group of $U_R$ is trivial,
then the corresponding real monodromy group is $\{(1)\}\subset\cS_R$.  In particular, if $P = 1$, then the corresponding
real monodromy group is $\{(1)\}\subset\cS_R$.
\end{theorem}
\begin{proof}
If the fundamental group of $U_R$ is trivial,
then all loops in $U_R$ are contractible
yielding only the identity permutation in the real monodromy group.
In particular, when $P=1$, then $U_R$ is an interval
where all loops are contractible.
\end{proof}

The system from Ex.~\ref{ex:Notes2.1/2.2} was specifically designed to have a nontrivial real monodromy group.  Since one often expects each corresponding set $U_R$ to be a cell thereby
having a trivial fundamental group,
one can often expect trivial
real monodromy groups for problems arising
in applications.

\begin{figure}[!b]
\begin{center}
	\includegraphics[scale = 0.25]{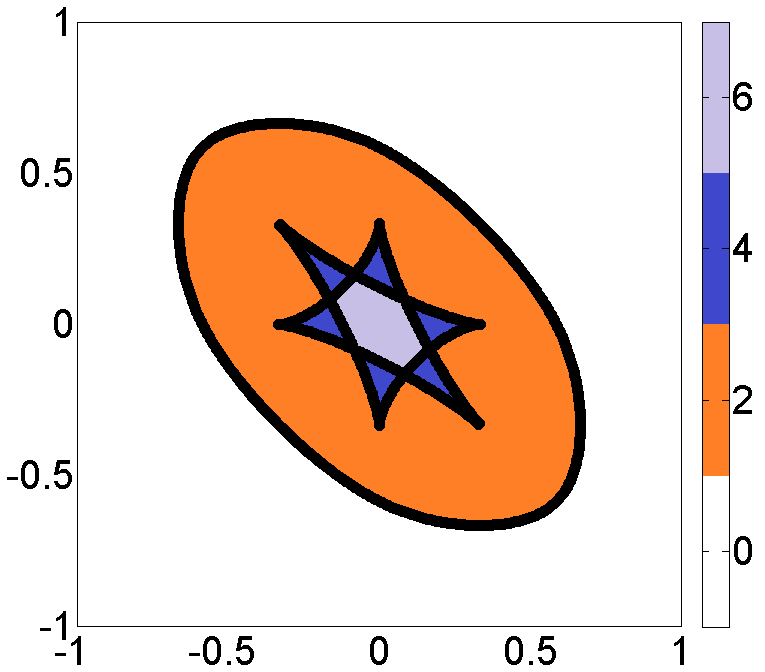}
    \end{center}
\caption{Parameter space $\bR^2$ colored by
number of nonsingular real solutions to \eqref{eq:Kuramoto}.}
    \label{Fig:Kuramoto}
\end{figure}

\begin{example}\label{Ex:Kuramoto}
The Kuramoto model \cite{Kuramoto}
provides a mathematical model of
synchronous behavior of coupled oscillators.
We consider $n = 3$ oscillators and
avoid the trivial rotation by fixing $\theta_3 = 0$.
For parameters $\omega = (\omega_1,\omega_2)$,
the steady-state solutions of the Kuramoto model satisfy
$$\left[\begin{array}{c}
\sin(\theta_1-\theta_2)+\sin(\theta_1-\theta_3) - 3\omega_1 \\
\sin(\theta_2-\theta_1)+\sin(\theta_2-\theta_3) - 3\omega_2
\end{array}\right] = 0.$$
We convert to a polynomial system
by taking $s_i = \sin(\theta_i)$ and $c_i = \cos(\theta_i)$, namely
\begin{equation}\label{eq:Kuramoto}
F(s_1,c_1,s_2,c_2;\omega_1,\omega_2) =
\left[\begin{array}{c}
(s_1 c_2 - c_1 s_2) + (s_1 c_3 - c_1 s_3) - 3\omega_1 \\
(s_2 c_1 - c_2 s_1) + (s_2 c_3 - c_2 s_3) - 3\omega_2 \\
s_1^2 + c_1^2 - 1 \\
s_2^2 + c_2^2 - 1 \\
\end{array}\right].
\end{equation}
Figure~\ref{Fig:Kuramoto} adapted from \cite[Fig.~2]{CossKuramoto} colors the parameters based
on the number of real solutions.
In particular, the connected set having $6$
real solutions and each of the six connected sets
having $4$ real solutions have trivial fundamental
group and thus trivial real monodromy group.
The connected set having $2$ real solutions
has a nontrivial fundamental group, but
one can check that this set also has a trivial
real monodromy group.
The reason for this will become
apparent in Ex.~\ref{Ex:Kuramoto2}.
Moreover, the (complex)
monodromy group
computed using \cite{MonodromyAlg} is $\cS_6$,
which is not a solvable group.

\end{example}

The condition that the number of nonsingular real solutions along loops in the real parameter space remain constant is a restriction
that both ensures a group structure
and is responsible for often having a
trivial real monodromy group.
The following provides
an illustration of this restriction.

\begin{example}\label{ex:BadRealGp}
Consider the parameterized polynomial system
\begin{equation}\label{eq:BadRealGp}
F(x;p) = \left[\begin{array}{c} (x_1^2-x_2^2-p_1)(x_1^2+p_1) \\ 2x_1x_2 - p_2 \end{array}\right] = 0
\end{equation}
which is a modification of the system considered in
Ex.~\ref{ex:Notes1.2} and~\ref{ex:Notes2.1/2.2}.
For $b = (-1,0)$, there are $R=4$ nonsingular
real solutions, namely:
$$
x^{(1)} = (1,0),\hspace{5mm} x^{(2)} = (-1,0),\hspace{5mm} x^{(3)} = (0,1),\hspace{5mm} x^{(4)} = (0,-1).
$$
In this case, $U_R = \{p_1 < 0\}\subset\bR^2$
which has a trivial fundamental group and thus the
real monodromy group is trivial.
Figure~\ref{Fig:BadRealGp} illustrates the decomposition
of the parameter space based on the number of real solutions.

Consider the loop $\gamma(t) = (-\cos(2\pi t),\sin(2\pi t))$
for $t\in[0,1]$ starting and ending at $b$.
If we only focus on the two solution paths
starting at $x^{(1)}$ and $x^{(2)}$,
these paths remain nonsingular over the loop and the endpoints
interchange as one would expect from
the real monodromy group in Ex.~\ref{ex:Notes2.1/2.2}.
The two solution paths
starting at $x^{(3)}$ and $x^{(4)}$
are nonsingular and real for $t\in[0,1/4)\cup(3/4,1]$,
nonsingular and nonreal for $t\in(1/4,3/4)$,
and at infinity for $t\in\{1/4,3/4\}$.

\begin{figure}[!ht]
\begin{center}
\includegraphics[scale = 0.5]{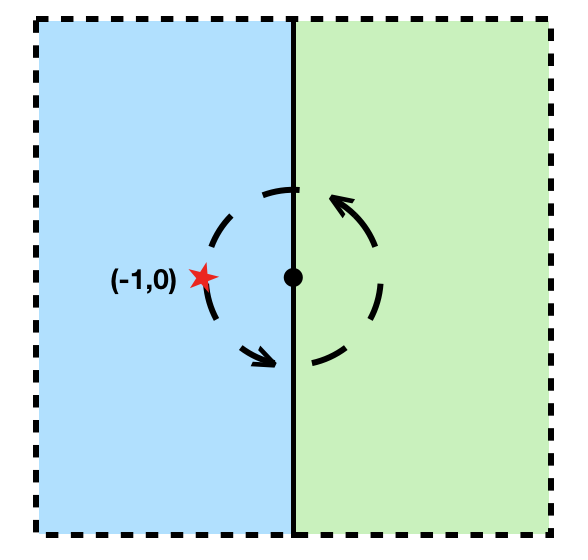}
\end{center}
\caption{Parameter space $\bR^2$, colored
by number of nonsingular real solutions
for $F(x;p)$ in \eqref{eq:BadRealGp}: $2$ on $\{p_1 > 0\}$ (green) and $4$ on $\{p_1 < 0\}$ (blue), base point $b = (-1,0)$, and loop $\gamma(t)$.}
\label{Fig:BadRealGp}
\end{figure}
\end{example}

Example~3.4 shows that important information about
the connections between some of the
real solutions can be obtained by
relaxing the requirement that {\em all}
real solutions remain nonsingular along
the path.  With this relaxation, one
may lose the group structure, but obtains a complete
picture of the interconnections between
the real solutions over a base point $b$,
which is described~next.

\subsection{Real monodromy structure}\label{Sec:RMG_def2}

Writing the monodromy action as a group of permutations
provides an efficient representation of the behavior
of the solutions.  When the action under consideration
may no longer be a group, we propose a structure
to encode the action called the {\em real monodromy structure}.

The following summarizes some sets of interest.

\begin{mydef}
For a nonnegative integer $R$, let
$\Pow(R)$ be the {\em power set}
of $\{1,\dots,R\}$ which is the set of all subsets of $\{1,\dots,R\}$.  Let $\OPow(R)$ be the {\em ordered power set}
of $\{1,\dots,R\}$ which is the set of all ordered subsets of $\{1,\dots,R\}$.
For $0\leq k\leq R$, let $\OPow_k(R)$
be the {\em $k$-ordered power set}
of $\{1,\dots,R\}$ which is the set of all ordered subsets of $\{1,\dots,R\}$ of size $k$.
Let $\IPow_k(R)$ be the subset of $\OPow_k(R)$
where the entries of the
elements are listed in increasing order.
\end{mydef}

\begin{example}\label{ex:Notes1.3/4}
To illustrate, we consider $R = 2$.  Thus,
\begin{itemize}
\item $\Pow(2) = \{\emptyset,\{1\},\{2\},\{1,2\}\}$
\hbox{~~and~~~~}$\OPow(2) = \{\emptyset,\{1\},\{2\},\{1,2\},\{2,1\}\}$;
\item $\OPow_0(2) = \{\emptyset\}$;
\item $\OPow_1(2) = \{\{1\},\{2\}\} = \IPow_1(2)$;
\item $\OPow_2(2) = \{\{1,2\},\{2,1\}\}$
\hbox{~~and~~~~}$\IPow_2(2) = \{\{1,2\}\}$.
\end{itemize}
In particular, for elements in the power set, order does not matter.  Order does matter
for elements in the ordered power set.
\end{example}

With this notation, we present
the real monodromy structure.

\begin{mydef}
For the real parameterized polynomial system $F(x;p)$,
fix a base point $b\in\bR^P$ and let
$x^{(1)},\dots,x^{(R)}\in\bR^N$ be
the $R$ nonsingular isolated solutions of $F(x;b) = 0$.
The {\em real monodromy structure} of $F$ at $b$
is a collection $\sG_\bullet = \{\sG_1,\dots,\sG_R\}$
of functions
$$\sG_k:{\IPow}_k(R) \rightarrow \Pow({\OPow}_k(R))$$
for $k = 1,\dots,R$ constructed as follows.
For each $\mathcal{Q} = \{q_1,...,q_k\} \in \IPow_k(R)$, $\sG_k(\mathcal{Q})$ is the collection of subsets of $\OPow_k(R)$ of the form $\{s_1,...,s_k\}\in\OPow_k(R)$
such that there exists a loop $\gamma\subset\bR^P$
starting and ending at $b$
where the solution path
of $F(x;p) = 0$ over $\gamma$ starting at $x^{(q_i)}$
is nonsingular and ends at $x^{(s_i)}$
for $i=1,\dots,k$.
\end{mydef}

The real monodromy group
describes how the set of all solutions
can be permuted whereas the real monodromy
structure describes how each subset
of solutions can be permuted.
Hence, the real monodromy group
is encoded in $\sG_R$.

\begin{example}\label{ex:RMS_group}
Consider representing the real monodromy group
$\cS_2$ from Ex.~\ref{ex:Notes2.1/2.2}
as a real monodromy structure.
First, we construct the map $\sG_1$.
Since both corresponding solutions
can trivially return to themselves
or connect to the other,
$\sG_1$ maps both $\{1\}$ and $\{2\}$
to $\{\{1\},\{2\}\}$.
Similarly,~$\sG_2$ maps $\{1,2\}$
to $\{\{1,2\},\{2,1\}\}$
since the pair can remain unchanged or be permuted.
For simplicity, we will write $\sG_\bullet = \{\sG_1,\sG_2\}$ as
\newpage
\begin{itemize}
\item $\mathcal{G}_1$
	\begin{itemize}
		\item $\{1\},\{2\} \mymapsto \{\{1\},\{2\}\}$
	\end{itemize}
\item $\mathcal{G}_2$
	\begin{itemize}
		\item $\{1,2\} \mymapsto \{\{1,2\},\{2,1\}\}$.
	\end{itemize}
\end{itemize}
\end{example}

\begin{example}\label{ex:BadRealGpMS}
For Ex.~\ref{ex:BadRealGp}
with $b = (-1,0)$, the real monodromy structure
is $\sG_\bullet = \{\sG_1,\dots,\sG_4\}$:
\begin{itemize}
\item $\mathcal{G}_1$
	\begin{itemize}
		\item $\{1\},\{2\} \mymapsto \{\{1\},\{2\}\}$
		\item $\{q_1\}\mymapsto \{\{q_1\}\}$ for all others
	\end{itemize}
\item $\mathcal{G}_2$
	\begin{itemize}
		\item $\{1,2\} \mymapsto \{\{1,2\},\{2,1\}\}$
		\item $\{q_1,q_2\} \mymapsto \{\{q_1,q_2\}\}$ for all others
	\end{itemize}
\item $\mathcal{G}_3$
	\begin{itemize}
		\item $\{q_1,q_2,q_3\} \mymapsto \{\{q_1,q_2,q_3\}\}$
	\end{itemize}
\item $\mathcal{G}_4$
	\begin{itemize}
		\item $\{q_1,q_2,q_3,q_4\} \mymapsto \{\{q_1,q_2,q_3,q_4\}\}$.
	\end{itemize}
\end{itemize}
This shows that there is an interconnection
between $x^{(1)}$ and $x^{(2)}$ as
described by the loop
$\gamma(t)$ in Ex.~\ref{ex:BadRealGp},
whereas anything involving
$x^{(3)}$ or $x^{(4)}$ must be trivial.
The function $\sG_4$ encodes
the triviality of the real monodromy
group.  Moreover, $\sG_\bullet$ does not
have a~group structure since $\sG_4$ is trivial
while $\sG_1$ and $\sG_2$ are not.
\end{example}

\begin{remark}
A nonsingular assembly mode change
means that there is a real nonsingular path between
solution $x^{(i)}$ and $x^{(j)}$ where $i\neq j$.
Hence, nonsingular assembly mode changes
are described by $\sG_1$, namely $i\neq j$
such that $\{j\}\in \sG_1(\{i\})$.
\end{remark}

A group $G\subset\cS_R$ is said to be {\em $k$-transitive}
if, for every $\mathcal{Q} = \{q_1,...,q_k\}\in \IPow_k(R)$ and $\mathcal{T} = \{t_1,...,t_k\}\in\OPow_k(R)$,
there exists $\sigma\in G$ such that
$\sigma(q_i) = t_i$ for $i = 1,\dots,k$.
Clearly, if~$G$ is $k$-transitive for $k>1$,
then $G$ is also $(k-1)$-transitive.
If $G$ is $1$-transitive, then $G$
is~called~{\em transitive}.

\begin{example}
The group $K_4\subset\cS_4$ in \eqref{eq:K4} is
transitive.  It is not $2$-transitive since
there does not exist $\sigma\in K_4$
which maps the ordered set $\{1,2\}$ to
the ordered set $\{1,3\}$.
\end{example}

The following extends the notion
of transitivity to the real monodromy structure.

\begin{mydef}
A real monodromy structure $\sG_\bullet=\{\sG_1,\dots,\sG_R\}$
is {\em $k$-transitive} if
$$\sG_k(\mathcal{Q}) = {\OPow}_k(R)\in\Pow({\OPow}_k(R))$$
for every $\mathcal{Q}\in \IPow_k(R)$.
\end{mydef}

Clearly, if $\sG_\bullet$ is $k$-transitive
for $k > 1$,
then $\sG_\bullet$ is also $(k-1)$-transitive.
Moreover, $\sG_\bullet$ is called {\em transitive}
if it is $1$-transitive.  Hence,
a real monodromy structure is transitive
if and only if
$\sG_1(\{i\}) = \{\{1\},\dots,\{R\}\} = \OPow_1(R)$
for all $i = 1,\dots,R$ meaning
that, for every $i,j\in \{1,\dots,R\}$,
there is a nonsingular solution path
starting at $x^{(i)}$ and ending at $x^{(j)}$.

\begin{example}
The real monodromy structure $\sG_\bullet$ in
Ex.~\ref{ex:RMS_group} is $2$-transitive and hence
$1$-transitive.  The real monodromy stucture $\sG_\bullet$ in
Ex.~\ref{ex:BadRealGpMS} is not transitive.
\end{example}

We conclude this section with an
illustrative description of computing
the real monodromy structure for
the Kuramoto model in \eqref{eq:Kuramoto}
which is dependent on 2 parameters.

\begin{example}\label{Ex:Kuramoto2}
For the $n = 3$ Kuramoto system $F$
in \eqref{eq:Kuramoto},
Ex.~\ref{Ex:Kuramoto} showed that every
real monodromy group is trivial
and the (complex) monodromy group is
the full symmetric group $\cS_6$.
The following computes
the real monodromy structure $\sG_\bullet$
at $\omega^* = (0,0)$,
which is nontrivial thereby
identifying interconnections between the
real solutions.  At $\omega^*$, we label the six
solutions as
$$
\begin{array}{c}
x^{(1)} = (0,1,0,1),\hspace{5mm}
x^{(2)} = (0,1,0,-1),\hspace{5mm}
x^{(3)} = (0,-1,0,1),\hspace{5mm}
x^{(4)} = (0,-1,0,-1), \\[0.05in]
x^{(5)} = \frac{1}{2}(\sqrt{3},-1,-\sqrt{3},-1),\hspace{5mm}
x^{(6)} = \frac{1}{2}(-\sqrt{3},-1,\sqrt{3},-1).
\end{array}
$$

First, one decomposes the parameter space
into connected components based on the number of
real solutions.  Figure~\ref{Fig:Kuramoto} was generated, for example, using~\cite{Discriminant}
via {\tt Bertini} \cite{BHSW:Bertini}.
Since the connected hexagonal region
containing $\omega^*$ has a trivial fundamental
group, it follows that both
$\sG_5$ and $\sG_6$ are trivial.
Moreover, since the
curved ``hexagram,'' i.e., ``six-sided star,'' region
consisting of the set
of parameter values with at least $4$ real solutions
also has a trivial fundamental group,
$\sG_3$ and $\sG_4$ must also be trivial.
Hence, we only need to consider $\sG_1$ and $\sG_2$.

To help with the bookkeeping, we fix a marked
point in each connected component with~$\omega^*$
being the one in its component
and place an ordering on the real solutions
over each marked point.  Then, in the
two parameter case, we identify each
of the finitely many
smooth segments of the boundaries of the connected
components.  Along such segments of the boundary,
there is a consistent identification of
the real solutions which are no longer
nonsingular at the boundary.
Thus, there is an equivalence along each smooth
segment of the boundary.  A similar statement
holds for smooth regions of the boundary
when there are more parameters.

From Fig.~\ref{Fig:Kuramoto}, there are
$18$ smooth boundary segments to
consider: $6$
for the region having $6$ real solutions
and two additional ones for each of
the $6$ regions having $4$ real solutions.  A homotopy is used to identify
the behavior of the solutions between
each boundary segment using the
consistent ordering from the marked point of the region.  Intermediate points
can be used to assist in this, which are
especially useful in nonconvex regions
to connect the marked points
as shown~in~Fig.~\ref{Fig:Kuramoto2}.

\begin{figure}[!t]
\begin{center}
	\includegraphics[scale = 0.35]{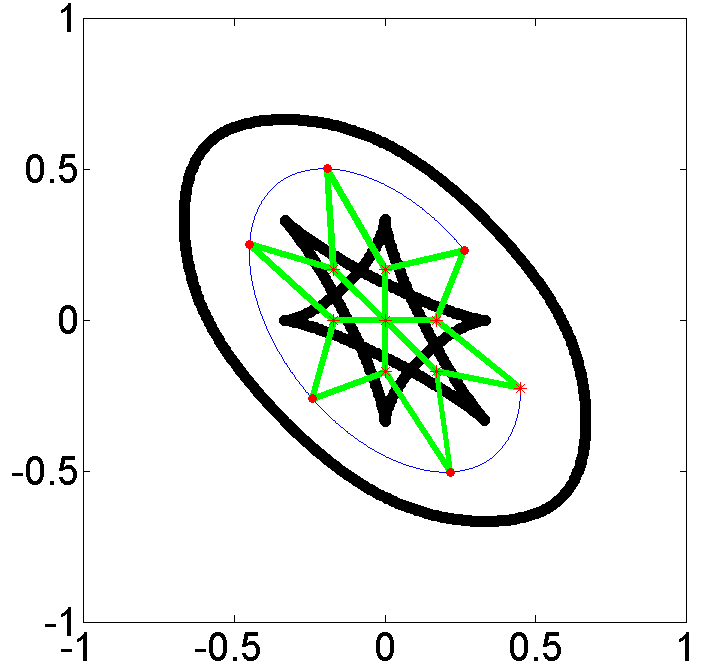}
    \end{center}
\caption{A plot of the boundaries of the connected components (black lines) for \eqref{eq:Kuramoto},
selected marked points (red star)
in each connected component, all intermediary
points (red dot), paths crossing the boundary (green lines), and connections between intermediate points
and the marked point of the component (blue curves).}
    \label{Fig:Kuramoto2}
\end{figure}

Finally, an analysis of the data generated
from traversing the boundaries produces all
of the possible interconnections between
the real solutions $x^{(1)},\dots,x^{(6)}$
at $\omega^*$.  In particular, this analysis
shows that the hexagram region consists of two large curved triangles: one pointing northwest and the other point southeast.
On the boundary of the northwest triangle, $x^{(6)}$ always becomes singular and merges with
one of $x^{(2)}$, $x^{(3)}$, or $x^{(4)}$ along
its three sides.
Similarly, for the boundary
of the southeast triangle, $x^{(5)}$ always becomes singular and merges with
one of $x^{(2)}$, $x^{(3)}$, or $x^{(4)}$ along
its three sides.
Since one must cross
both triangular boundaries, to possibly
have a nontrivial action,
we know that $\sG_1$
and $\sG_2$ applied to any set involving
either $5$ or $6$ is trivial.

Furthermore, this analysis shows that the solution sheet corresponding to $x^{(1)}$
over the colored
region of the parameter space
in Fig.~\ref{Fig:Kuramoto} is nonsingular.
In particular, this explains
why the real monodromy
group computed in
Ex.~\ref{Ex:Kuramoto} for the connected component
colored orange in Fig.~\ref{Fig:Kuramoto}
having~2 real solutions is trivial
even though the fundamental group
for this component is not trivial.
Hence, $\sG_1(\{1\}) = \{\{1\}\}$
but it is different than~$x^{(5)}$ and $x^{(6)}$ in
that it can be included in nontrivial $\sG_2$ action.
In fact, since it is possible to move
from $\omega^*$ into
this orange component
using nonsingular paths starting at $x^{(1)}$
and one of $x^{(2)}$, $x^{(3)}$, or $x^{(4)}$,
we have that $\sG_1$ is transitive on
$\{2\}$, $\{3\}$, and~$\{4\}$
and $\sG_2$ is transitive on $\{1,2\}$,
$\{1,3\}$, and $\{1,4\}$.
This is summarized in the following:
\begin{itemize}
\item $\mathcal{G}_1$
	\begin{itemize}
		\item $\{2\},\{3\},\{4\} \mymapsto \{\{2\},\{3\},\{4\}\}$
		\item $\{q_1\}\mymapsto \{\{q_1\}\}$ for all others
	\end{itemize}
\item $\mathcal{G}_2$
	\begin{itemize}
		\item $\{1,2\},\{1,3\},\{1,4\} \mymapsto \{\{1,2\},\{1,3\},\{1,4\}\}$
		\item $\{q_1,q_2\} \mymapsto \{\{q_1,q_2\}\}$ for all others.
	\end{itemize}
\end{itemize}
Therefore, the real monodromy structure
provides three distinct groups
of solutions which have similar properties:
$\{x^{(1)}\}$, $\{x^{(2)},x^{(3)},x^{(4)}\}$,
and $\{x^{(5)},x^{(6)}\}$.
\end{example}

\section{Real monodromy of the 3RPR mechanism}\label{Sec:3RPR}

The 3RPR mechanism, as shown in Fig.~\ref{Fig:mechanism}, is a well-known
mechanism that has a nonsingular assembly mode
change \cite{Gosselin,Hayes,Zein,Innocenti,Macho,Bonev,Husty}.
Thus, with an appropriate choice
of base point,
$\sG_1$ of the real monodromy
structure for this base point will be nontrivial.
In this section, we compute the complete
real monodromy structure $\sG_\bullet$
using a base point from \cite{Husty}
when one of the leg lengths is fixed
and the other two legs are free to change lengths.

As shown in Fig.~\ref{Fig:3RPR_Mechanism}, let the leg lengths be $\ell_1$, $\ell_2$,
and $\ell_3$.  For simplicity, we will
consider the squares of the leg lengths,
namely $c_i = \ell_i^2$.
In the fixed frame, set the three anchors of the
three legs, respectively, at
$(0,0)$, $(A_2,0)$, and $(A_3,B_3)$.
In the moving frame, set the three connections
of the three legs, respectively, attached to the
triangle at $P_1 = (0,0)$, $P_2 = (a_2,0)$, and $P_3 = (a_3,b_3)$.
Following the case studied in \cite{Husty},
we take the following constants:
\begin{equation}\label{eq:mechConstants}
a_2 = 14, \hspace{2mm} a_3 = 7, \hspace{2mm} b_3 = 10, \hspace{2mm} A_2 = 16, \hspace{2mm} A_3 = 9, \hspace{2mm} B_3 = 6, \hspace{2mm} c_3 = \ell_3^2 = 100.
\end{equation}
For the parameters $c = (c_1,c_2)$,
we take the base point, corresponding
to the ``home'' position,
to be $c^* = (75,70)$ as in \cite{Husty}
where the mechanisms
satisfying this setup
are shown in Fig.~\ref{fig:six_solutions}.

\begin{figure}[t!]
	\centering
	\includegraphics[scale = 0.6]{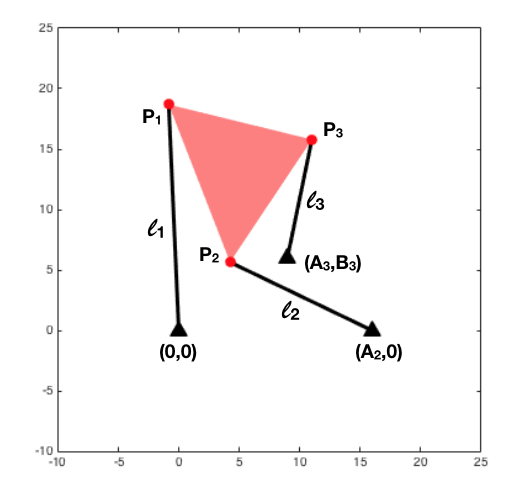}
	\caption{An illustration of a 3RPR mechanism.}\label{Fig:3RPR_Mechanism}
\end{figure}

The variables $(p,\phi) = (p_1,p_2,\phi_1,\phi_2)$
of the polynomial system represent
the relative position and rotation between
fixed frame and moving frame, respectively.
The polynomials constrain
the rotation $(\phi_1,\phi_2)$
to a point on the unit circle
and describe the three leg constraints:
$$F(p,\phi;c) =
\left[\begin{array}{l}
\phi_1^2 + \phi_2^2 - 1 \\
p_1^2 + p_2^2 - 2(a_3p_1 + b_3p_2)\phi_1 + 2(b_3p_1 - a_3p_2)\phi_2 + a_3^2 + b_3^2 - c_1 \\
p_1^2 + p_2^2 - 2A_2p_1 +
2((a_2-a_3)p_1 - b_3p_2 + A_2a_3 - A_2a_2)\phi_1 \\
\hspace{10mm} +~2(b_3p_1 + (a_2-a_3)p_2 - A_2b_3)\phi_2
+ (a_2-a_3)^2 + b_3^2 + A_2^2 - c_2 \\
p_1^2 + p_2^2 - 2(A_3 p_1 + B_3 p_2) + A_3^2 + B_3^2 - c_3
\end{array}\right].$$
At the ``home'' position $c^* = (75,70)$,
the system $F(p,\phi;c^*) = 0$ has $6$
nonsingular real solutions
which are assigned labels
$x^{(1)},\dots,x^{(6)}$ in Fig.~\ref{fig:six_solutions}.
The remaining part of this section
describes computing the
real monodromy structure $\sG_\bullet=\{\sG_1,\dots,\sG_6\}$
for $F$ at $c^*$
where we utilize {\tt Bertini} \cite{BHSW:Bertini}
to perform the homotopy continuation computations.

\begin{figure}[!t]
\centering
\begin{tabular}{ccc}
  \includegraphics[scale = 0.25]{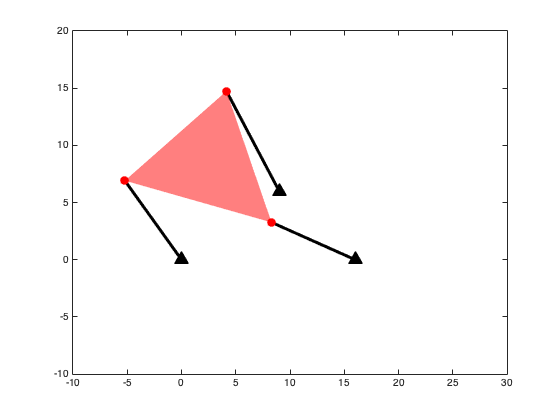} &
  \includegraphics[scale = 0.25]{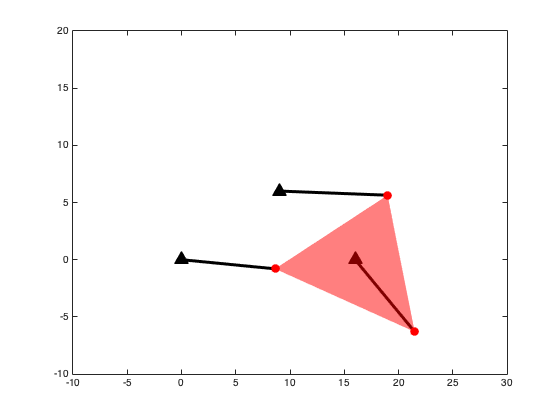} &
  \includegraphics[scale = 0.25]{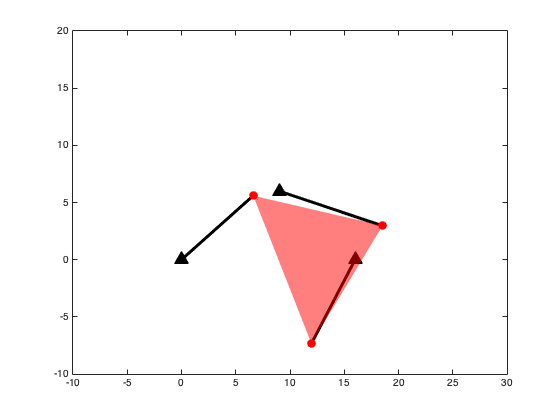} \\
  $x^{(1)}$ & $x^{(2)}$ & $x^{(3)}$ \\[0.1in]
  \includegraphics[scale = 0.25]{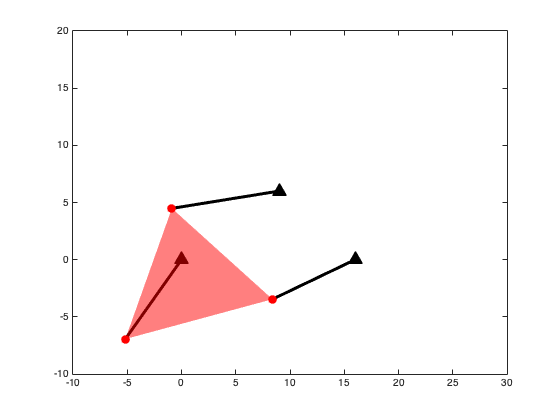} &
  \includegraphics[scale = 0.25]{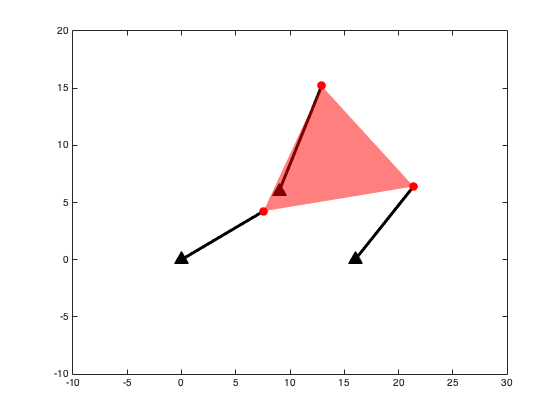} &
  \includegraphics[scale = 0.25]{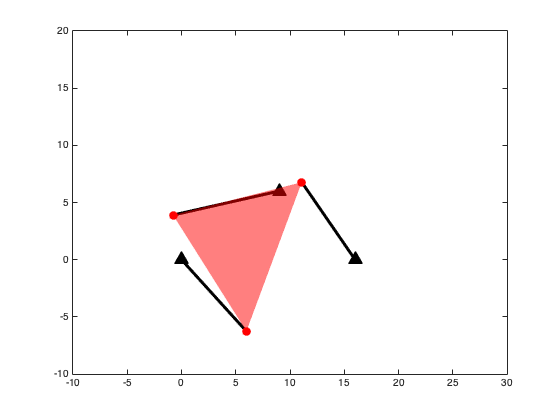}\\
    $x^{(4)}$ & $x^{(5)}$ & $x^{(6)}$
\end{tabular}
\caption{The $6$ solutions
to $F(p,\phi;c^*) = 0$.}
\label{fig:six_solutions}
\end{figure}

First, we compute the boundaries of the
subsets of $\bR^2$ where the number of real
solutions change using \cite{Discriminant}.  Figure~\ref{Fig:3RPR_Discriminant}
colors the regions in $c = (c_1,c_2)\in \bR^2$
having $0$, $2$, $4$, and $6$ real solutions,
where $c^* = (75,70)$ lies in the unique
connected component having $6$ real solutions.
In particular, since this component
has a trivial fundamental group, Theorem~\ref{Thm:OnePar_RMG} concludes
that the real monodromy group is trivial.
It follows that $\sG_5$ and $\sG_6$ are also trivial.
We note that the (complex) monodromy group
is $\cS_6$ showing there is no complex structure
in the solutions encoded by the~monodromy~group.

\begin{figure}[t!]
\begin{subfigure}{.5\textwidth}
  \centering
  \includegraphics[scale = 0.31]{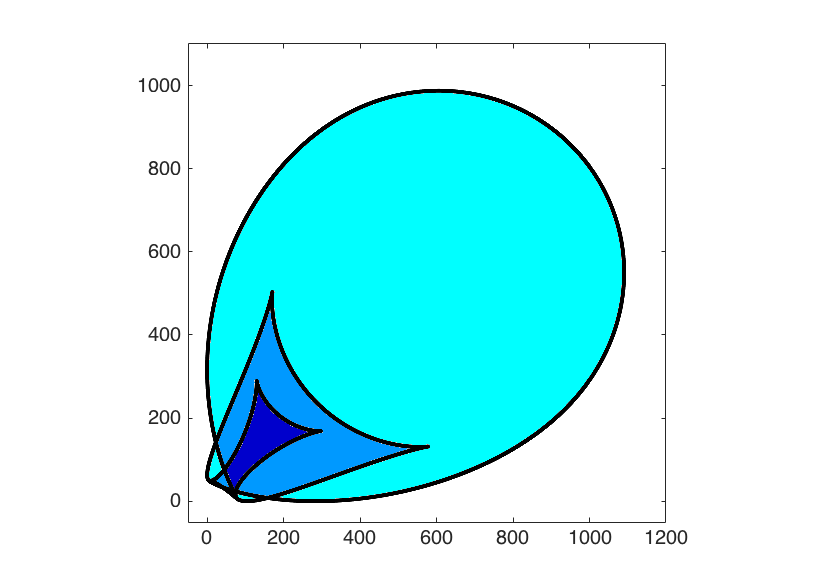}
  \caption{}
  \label{fig:disclocus1}
\end{subfigure}%
\begin{subfigure}{.5\textwidth}
  \centering
  \includegraphics[scale = 0.31]{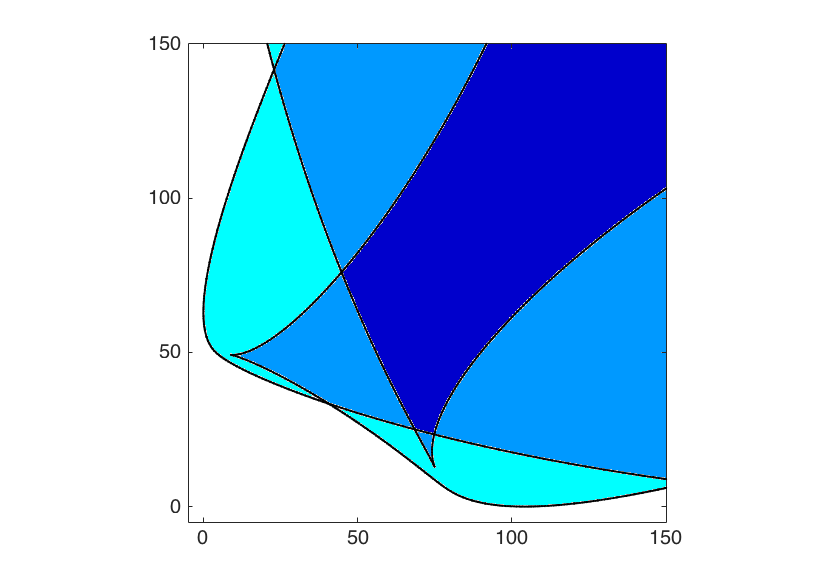}
  \caption{}
  \label{fig:disclocus2}
\end{subfigure}
\caption{Regions of the parameter space $c = (c_1,c_2)$ colored by the number of real solutions
where (a) is the full view and (b) is a zoomed in view of the lower left corner. The navy blue region has~6 real solutions, the grey blue region has 4 real solutions, the baby blue region has 2 real solutions,
and the white region has 0 real solutions.}\label{Fig:3RPR_Discriminant}
\end{figure}

To help with the bookkeeping,
we fix a marked point in each connected
component and select additional intermediate points
to simplify the computation of loops
as shown in Fig.~\ref{Fig:3RPR_basepoints}.
To compute all possible loops, we need
to transverse between
the connected components.
Since there is an equivalence
by passing through smooth regions of the
boundary, there are only finitely many
possible loops of interest.
For example, leaving the navy blue region
having 6 real solutions
and entering into the largest grey blue
connected component that touches it
has at most three different outcomes
obtained by crossing through the three different
smooth segments of the boundary.
The intermediate points facilitate
moving the solutions to the marked point
to have consistent~ordering.

\begin{figure}[t!]
\begin{subfigure}{.5\textwidth}
  \centering
  \includegraphics[scale = 0.31]{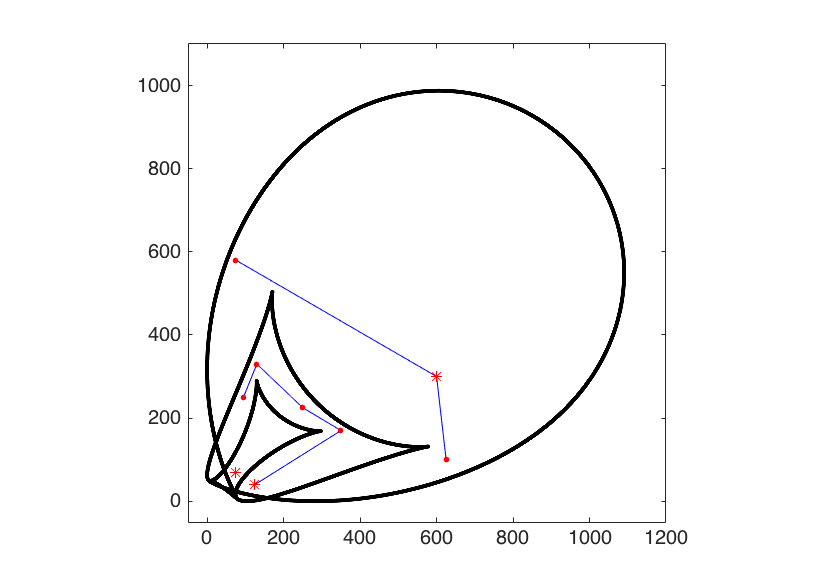}
  \caption{}
  \label{fig:basepoints1}
\end{subfigure}%
\begin{subfigure}{.5\textwidth}
  \centering
  \includegraphics[scale = 0.31]{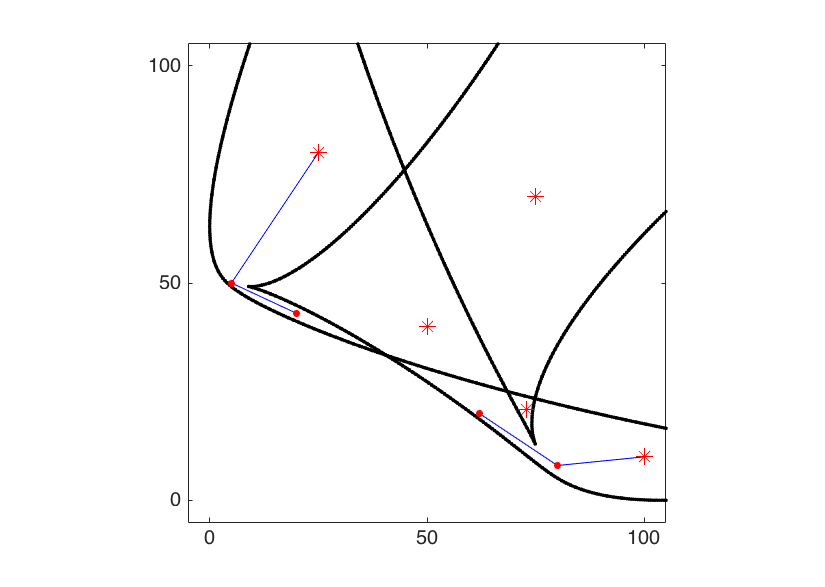}
  \caption{}
  \label{fig:basepoints2}
\end{subfigure}
\caption{Selected marked points (red star)
in each connected component and all intermediary points (red dot) in (a) full view and (b) zoomed in view of the lower left corner.}\label{Fig:3RPR_basepoints}
\end{figure}

By tracking solutions along all possible
loops and carefully keeping track of those
that remain nonsingular, we obtain
the following real monodromy structure
where $\sG_5$ and $\sG_6$ are trivial:
\begin{itemize}
\item $\mathcal{G}_1$
	\begin{itemize}
		\item $\{1\},\{2\},\{3\} \mymapsto \{\{1\},\{2\},\{3\}\}$
		\item $\{4\},\{5\},\{6\} \mymapsto \{\{4\},\{5\},\{6\}\}$
	\end{itemize}
\item $\mathcal{G}_2$
	\begin{itemize}
		\item $\{1,4\},\{1,5\},\{1,6\},\{2,5\},\{2,6\},\{3,4\},\{3,5\} \mymapsto \left\{
\begin{array}{c}		
		\{1,4\},\{1,5\},\{1,6\},\{2,5\},\\
		\{2,6\},\{3,4\},\{3,5\}\end{array}\right\}$
		\item $\{1,3\},\{2,3\} \mymapsto \{\{1,3\},\{2,3\}\}$
		\item $\{4,6\},\{5,6\} \mymapsto \{\{4,6\},\{5,6\}\}$
        \item $\{q_1,q_2\} \mymapsto \{\{q_1,q_2\}\}$ for all others
	\end{itemize}
\item $\mathcal{G}_3$
	\begin{itemize}
		\item $\{1,4,6\},\{1,5,6\},\{2,5,6\} \mymapsto \{\{1,4,6\},\{1,5,6\},\{2,5,6\}\}$
		\item $\{1,3,6\},\{2,3,6\} \mymapsto \{\{1,3,6\},\{2,3,6\}\}$
		\item $\{3,4,6\},\{3,5,6\} \mymapsto \{\{3,4,6\},\{3,5,6\}\}$
\item $\{q_1,q_2,q_3\} \mymapsto \{\{q_1,q_2,q_3\}\}$ for all others

	\end{itemize}
\item $\mathcal{G}_4$
	\begin{itemize}
		\item $\{1,3,4,6\},\{1,3,5,6\},\{2,3,5,6\} \mymapsto \{\{1,3,4,6\},\{1,3,5,6\},\{2,3,5,6\}\}$
\item $\{q_1,q_2,q_3,q_4\} \mymapsto \{\{q_1,q_2,q_3,q_4\}\}$ for all others

	\end{itemize}
\end{itemize}

Figures~\ref{Fig:3RPR_Path}
and~\ref{Fig:3RPR_MechanismPath}
illustrate a nonsingular assembly mode change
between $x^{(4)}$ and $x^{(5)}$.

\begin{figure}[b!]
	\centering
	\includegraphics[scale = 0.34]{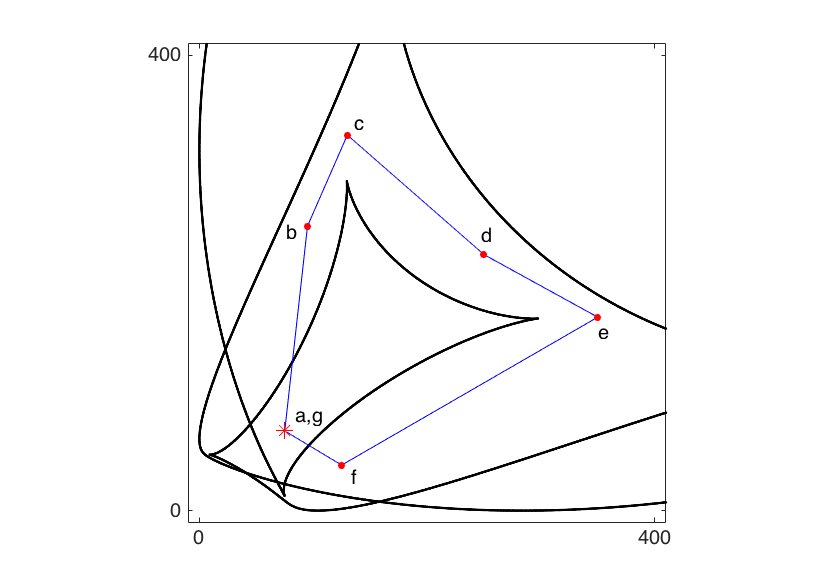}
	\caption{A loop starting and ending at $c^* = (75,70)$ that yields a nonsingular assembly mode change.
	The corresponding linkage pose for the labeled
	points are shown in Fig.~\ref{Fig:3RPR_MechanismPath}.}\label{Fig:3RPR_Path}
\end{figure}

\begin{figure}[b!]
\centering
	\includegraphics[scale = 0.55]{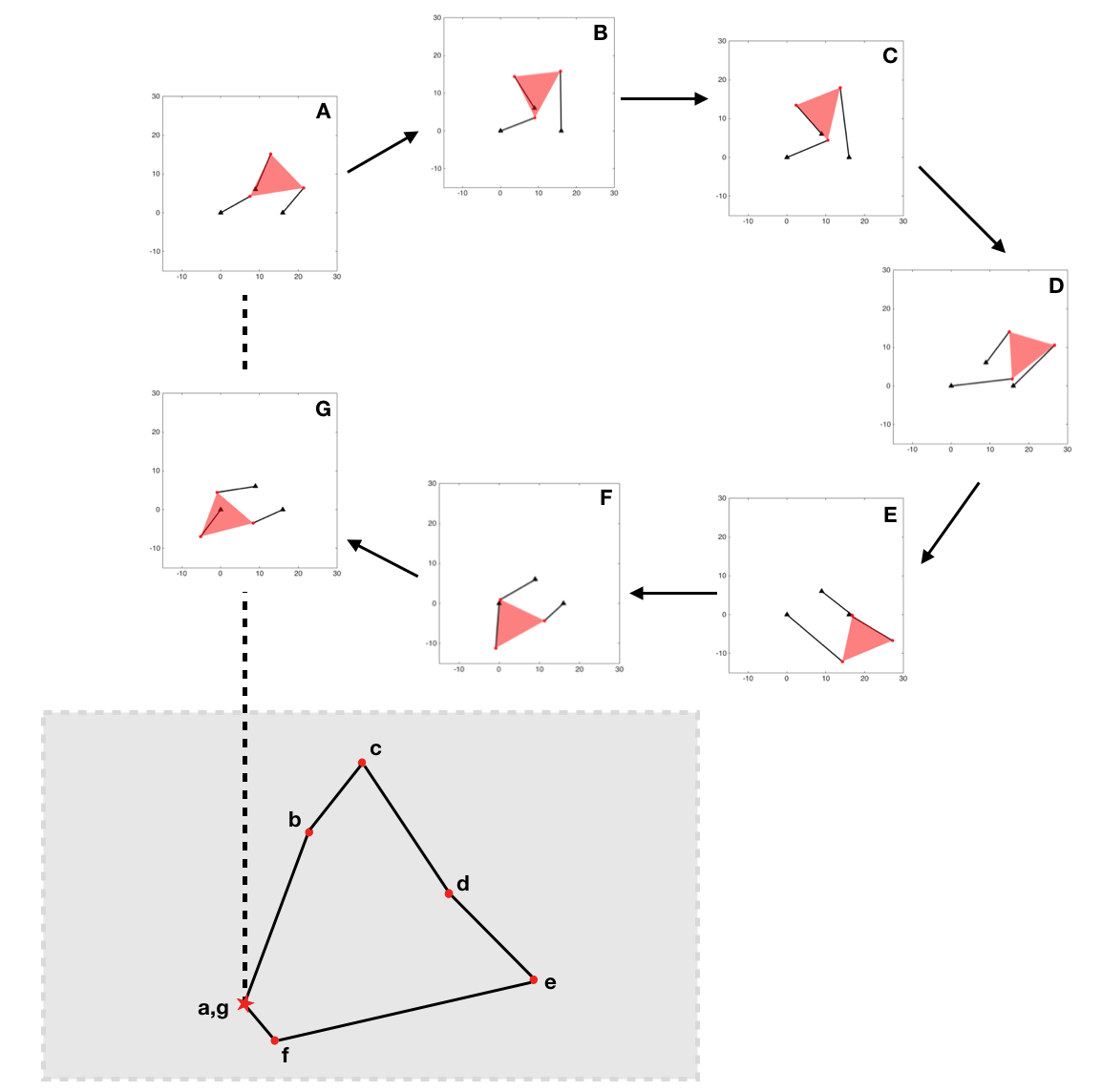}
\caption{Solutions along a loop that
yields a nonsingular assembly mode change.}
\label{Fig:3RPR_MechanismPath}
\end{figure}

The real monodromy structure $\sG_\bullet$
identifies that the real solutions arise in two groups
of three solutions coinciding with
the results in \cite{Husty}
which showed that there are two disjoint
path-connected components.
In fact, in light of this constraint,
$\sG_1$ shows that all possible
nonsingular assembly mode changes
can occur.
The real monodromy structure
$\sG_\bullet$ provides additional information
beyond nonsingular assembly mode changes
by considering how other
subsets of solutions can be interchanged.
For example, $\sG_4$ shows that
it is possible to interchange
two solutions $x^{(1)}$ and~$x^{(2)}$
while having three solutions
$x^{(3)}$, $x^{(5)}$, and $x^{(6)}$
return to themselves
with all four paths remaining real and nonsingular.
It is also possible to interchange
two pairs of solutions $x^{(1)}$
and~$x^{(2)}$, and $x^{(4)}$ and $x^{(5)}$
while having two solutions
$x^{(3)}$ and $x^{(6)}$
return to themselves
with all four paths remaining real and nonsingular.

\section{Conclusion}\label{Sec:Conclusion}

The (complex) monodromy group is a classically used
invariant in algebraic geometry
to study the structure of solutions to a parameterized
system of polynomial equations.
Since many applications involve working with real solution sets over real parameter spaces,
an extension of the monodromy action computations to the real numbers is needed.
A naive extension is to consider loops where all
real solutions stay real and nonsingular along the
solution path yielding the real monodromy group.
However, this is very restrictive and is
often trivial.  Thus,
we propose a real monodromy structure that gives
tiered information on the monodromy actions for the real solutions.  This enables useful
structural information to be obtained and circumvents the restrictiveness of the naive extension
by relaxing the condition that {\em all}
real paths remain nonsingular.

The real monodromy structure for the
3RPR mechanism allowing two legs to change length
describes how the solutions can interchange
thereby providing a complete mathematical generalization of nonsingular assembly mode changes.
This information can be useful, for example,
in calibration.  If no real solutions
can interchange, i.e., the real monodromy
structure is trivial, then returning to the ``home''
position avoiding singularities
will always yield the same pose.
However, if the real monodromy structure
is not trivial, then it describes all possible
interconnections between
poses over the ``home'' position.
Future work includes computing
real monodromy structures for Stewart-Gough platforms.

\section{Acknowledgments}\label{Sec:Acknowledgments}

The authors thank Charles Wampler for helpful discussions on kinematics and 3RPR mechanism.

\newcommand{\noopsort}[1]{}\def\cprime{$'$}

\end{document}